\documentclass[a4paper,11pt]{article}

\usepackage{amsmath, amssymb, amsthm}
\usepackage{hyperref}

\usepackage{enumerate}
\usepackage{bm}
\usepackage{graphicx}

\usepackage{pgf,tikz}
\usetikzlibrary{arrows}

\usepackage{color}

\usepackage{algorithm, algorithmic}

\newcommand{\mS}{\mathcal{S}}
\newcommand{\mP}{\mathcal{P}}
\newcommand{\mL}{\mathcal{L}}
\newcommand{\mI}{\mathrm{I}}
\newcommand{\mV}{\mathcal{V}}
\newcommand{\mN}{\mathcal{N}}
\newcommand{\mH}{\mathcal{H}}

\newcommand{\mG}{\mathcal{G}}
\newcommand{\mO}{\mathcal{O}}

\newcommand{\dist}{\mathrm{d}}

\newcommand{\HJ}{\mathsf{HJ}}

\newcommand{\m}[1]{\mathcal{#1}}

\mathchardef\mh="2D

\newtheorem{thm}{Theorem}[section]

\newtheorem{lem}[thm]{Lemma}
\newtheorem{cor}[thm]{Corollary}
\newtheorem{remark}[thm]{Remark}

\begin{document}

\title{Characterizations of the Suzuki tower near polygons}
\author{Anurag Bishnoi and Bart De Bruyn}
\maketitle

\begin{abstract}
In recent work, we constructed a new near octagon $\mathcal{G}$ from certain involutions of the finite simple group $G_2(4)$ and showed a correspondence between the Suzuki tower of finite simple groups, $L_3(2) < U_3(3) < J_2 < G_2(4) < Suz$, and the tower of near polygons, $\mathrm{H}(2,1) \subset \mathrm{H}(2)^D \subset \mathsf{HJ} \subset \mathcal{G}$. Here we characterize each of these near polygons (except for the first one) as the unique near polygon of the given order and diameter containing an isometrically embedded copy of the previous near polygon of the tower. In particular, our characterization of the Hall-Janko near octagon $\HJ$ is similar to an earlier characterization due to Cohen and Tits who proved that it is the unique regular near octagon with parameters $(2, 4; 0, 3)$, but instead of regularity we assume existence of an isometrically embedded dual split Cayley hexagon, $\mathrm{H}(2)^D$. We also give a complete classification of near hexagons of order $(2, 2)$ and use it to prove the uniqueness result for $\mathrm{H}(2)^D$.
\end{abstract}

\bigskip \noindent \textbf{Keywords:} near polygon, generalized polygon, Suzuki tower\\ 
\textbf{MSC2010:} 05E18, 51E12, 51E25

\section{Introduction} \label{sec1}

Near polygons, introduced by Shult and Yanushka \cite{Sh-Ya}, form an important class of point-line geometries. 
They are related to polar spaces \cite{Cameron}, distance-regular graphs \cite{BCN}, finite simple groups \cite{Wilson}, and they have a rich theory of their own \cite{bdb-book}. 
While dual polar spaces and generalized polygons (which can be seen as a special class of near polygons) are related to classical groups and exceptional groups of Lie type, there are also some near polygons corresponding to sporadic simple groups. 
One of the most famous examples is the Hall-Janko near octagon $\HJ$ (also known as the Cohen-Tits near octagon \cite{Br}) constructed by Cohen  \cite{Co} using the conjugacy class of $315$ central involutions of the Hall-Janko sporadic simple group $J_2$. 
The collinearity graph of $\HJ$ is a distance-regular graph with intersection array $\{10, 8, 8, 2; 1, 1, 4, 5\}$ which is uniquely determined by these intersection numbers, or equivalently, $\HJ$ is the unique regular near octagon with parameters $(2,4; 0, 3)$ (cf.~\cite[Thm.~3]{Co-Ti}). 
It is well known that $\HJ$ contains subgeometries isomorphic to the dual split Cayley hexagon of order $(2,2)$, henceforth denoted by $\mathrm{H}(2)^D$. 
One of our main results in this paper is an alternative characterization of $\HJ$ where instead of assuming regularity with parameters $(2, 4; 0, 3)$, we only assume the order $(2,4)$ and the existence of an isometrically embedded $\mathrm{H}(2)^D$ (see Theorem \ref{thm:HJ}). 

In \cite{ab-bdb:2} we constructed a new near octagon $\mG$ of order $(2, 10)$ which has its full automorphism group isomorphic to $G_2(4){:}2$ and proved that $\mG$ contains $\HJ$ as an isometrically embedded subgeometry. 
Our main construction in \cite{ab-bdb:2} was inspired by the construction of $\HJ$ given by Cohen \cite{Co} where the points of $\HJ$ are the $315$ central involutions of the group $J_2$, and all the lines are the triples $\{x, y, xy\}$ formed by taking two distinct commuting central involutions $x$ and $y$. 
We recall the construction of $\mG$ here:

\bigskip \noindent
\textit{Let $C$ be the conjugacy class $2A$ \footnote{see \url{http://brauer.maths.qmul.ac.uk/Atlas/v3/exc/G24/} for the notation and description of this class} consisting of  the $4095$ central involutions of the group $G = G_2(4){:}2$.
Then $\mG$ is the point-line geometry with point set $C$ whose lines are all $15015$ three element subsets $\{x, y, xy\}$ of $C$ where $x$ and $y$ are distinct commuting involutions in $C$ satisfying $[G : N_G(\langle x, y \rangle)] \in \{1365, 13650\}$.}

\bigskip The Suzuki tower (a name coined by Tits, as it has been mentioned in \cite{Leemans05})
is the sequence of five finite simple groups $L_3(2) < U_3(3) < J_2 < G_2(4) < Suz$ where each group (except the last one) is a maximal subgroup of the next group in the sequence. 
The five groups in the Suzuki tower correspond to five vertex-transitive graphs $\Gamma_0$, $\Gamma_1$, $\Gamma_2$, $\Gamma_3$, $\Gamma_4$ where the automorphism groups of the $\Gamma_i$'s are $L_3(2){:}2$, $U_3(3){:}2$, $J_2{:}2$, $G_2(4){:}2$ and $Suz{:}2$, respectively \cite{Suzuki}. 
Using the properties of $\mG$ that we have described in \cite{ab-bdb:2}, we gave a correspondence between the Suzuki tower and the tower of four near polygons $\mathrm{H}(2,1) \subset \mathrm{H}(2)^D \subset \HJ \subset \mG$, where $\mathrm{H}(2,1)$ is the unique generalized hexagon of order $(2,1)$ (which is the point-line dual of the incidence graph of the Fano plane). 
The first four simple groups in the Suzuki tower act as automorphism groups of the corresponding near polygons. 
Moreover, we were able to construct all of the graphs $\Gamma_i$'s using these near polygons. 
We note that there are various other geometries associated with the Suzuki tower that have been studied before, see for example \cite{Ne82}, \cite{S92} and \cite{Leemans05}. 

For all $i \in \{1, 2, 3, 4\}$ the graph $\Gamma_i$ is strongly regular and contains $\Gamma_{i-1}$ as a local graph. In fact, it has been proved in \cite{Pa} that for $i \in \{ 2, 3 \}$ the graph $\Gamma_i$ is the unique connected graph which is locally $\Gamma_{i-1}$, while there exist two connected graphs which are locally $\Gamma_3$, the graph $\Gamma_4$ and the graph $3\Gamma_4$ constructed by Soicher \cite{S93} which is a $3$-fold antipodal cover of the Suzuki graph $\Gamma_4$. In \cite{BFS93}, it was proved that there are up to isomorphism three connected graphs that are locally $\Gamma_0$, one of which is the $U_3(3)$-graph $\Gamma_1$. In the present paper, we give some characterizations of the geometries in the ``Suzuki tower of near polygons'' by proving the following three theorems.

\begin{thm}
\label{thm:H2}
The dual split Cayley hexagon of order $(2,2)$ is the unique near hexagon of order $(2,2)$ that contains the generalized hexagon $\mathrm{H}(2,1)$ as an isometrically embedded subgeometry. 
\end{thm}

\begin{thm}
\label{thm:HJ}
The Hall-Janko near octagon is the unique near octagon of order $(2,4)$ that contains the dual split Cayley hexagon of order $(2,2)$ as an isometrically embedded subgeometry. 
\end{thm}

\begin{thm}
\label{thm:G24}
The $G_2(4)$ near octagon is the unique near octagon of order $(2,10)$ that contains the Hall-Janko near octagon as an isometrically embedded subgeometry. 
\end{thm}

\begin{remark}
\normalfont
It is known that the second constituent of the $\HJ$-graph $\Gamma_2$ is isomorphic to the distance-$2$ graph of the generalized hexagon $\mathrm{H}(2)^D$. 
Similarly, the second constituent of the $G_2(4)$-graph $\Gamma_3$ is isomorphic to the distance-2 graph of the Hall-Janko near octagon $\HJ$. 
It has been communicated to us by one of the referees that a similar result holds for the near octagon $\m G$. 
Let $\Gamma = 3\Gamma_4$ be the triple cover of the Suzuki graph $\Gamma_4$ as constructed by Soicher \cite{S93}. Then $G = \mathrm{Aut}(\Gamma) \cong 3 \cdot Suz{:}2$ acts distance-transitively on the vertices and the vertices split into orbits of length $1, 416, 4095, 832, 2$ under the action of the stabilizer $G_x \cong G_2(4){:}2$ of a fixed vertex $x$ of $\Gamma$. Let $V$ be the set of $4095$ vertices at distance $2$ from $x$ in $\Gamma$. There exists a bijective correspondence between $V$ and the involution class $2A$ of $G_x$ (i.e., the points of the near octagon $\m G$), see Pasechnik \cite[Proposition 3]{Pa}. The collinearity graph of $\m G$, which uniquely determines the geometry $\m G$, can be constructed from $V$ by taking the edges as pairs $\{u, v\} \subset V$ whose stabilizer has index $4095$ or $40950$ in $G_x$. The subgraph of $\Gamma$ induced on $V$ is isomorphic to the graph defined on the point set of $\mathcal{G}$, by calling two vertices adjacent whenever they lie at distance 2 and have a unique common neighbour. These claims can be verified with the aid of the permutation representation of $3 \cdot Suz{:}2$ on the $5346$ vertices of $\Gamma$ given in the ATLAS \footnote{see \url{http://brauer.maths.qmul.ac.uk/Atlas/v3/permrep/3Suzd2G1-p5346B0}}.
\end{remark}

\begin{remark}
\normalfont
We do not know if there is a near polygon which corresponds to $Suz$, but certainly the involution geometry of $Suz$ studied in the literature \cite{Yoshiara88, Bardoe96} is not a near polygon; we can directly see from the suborbit diagram \cite[Fig. 1]{Bardoe96} that there are point-line pairs $(p, L)$ in this geometry where every point of the line $L$ is at the same distance $4$ from $p$. 
However, this involution geometry is a near $9$-gon in the sense of \cite[Sec. 6.4]{BCN}.
Indeed, it is clear from the suborbit diagram \cite[Fig. 1]{Bardoe96} that for every point-line pair $(p, L)$ with $\dist(p, L) < 4$, there is a unique point on $L$ nearest to $p$.
Similarly, one can verify that the involution geometry of the Conway group $Co_1$ \cite[Fig. 1]{Bardoe99} which contains $Suz$ is a near $11$-gon. 
The techniques involving valuations of near polygons used in this paper do not work for near $(2d+1)$-gons, and thus we are unable to give similar characterizations of these involution geometries. 
We do not know if there are more near polygons hiding in larger groups that can extend the Suzuki tower of near polygons (see \cite{S92} for a possible extension of the Suzuki tower of groups). 
\end{remark}

This paper is organised as follows. In Section \ref{sec:Prelim} we give the basic properties of near polygons and their valuations that we will use in our proofs.
In Section \ref{sec:H2} we first prove that up to isomorphism there are only three near hexagons of order $(2, 2)$, the two generalized hexagons of order $(2,2)$ (split Cayley hexagon and its dual) and the product near polygon $\mathbb{L}_3 \times \mathbb{L}_3 \times \mathbb{L}_3$.
Then we show that out of these three near hexagons only the dual split Cayley hexagon contains $\mathrm{H}(2,1)$ as a subgeometry, thus proving Theorem \ref{thm:H2}. 

Sections \ref{sec:HJ} and \ref{sec:G24} are then devoted to proving Theorems \ref{thm:HJ} and \ref{thm:G24}, respectively. 
The main tool that we use in these proofs is the theory of valuations of near polygons introduced by De Bruyn and Vandencasteele in \cite{bdb-pvdc}.
This is a purely combinatorial tool which has since been developed and used to obtain several classification results for near polygons (see \cite{bdb-val}  for a survey). 
In \cite{ab-bdb:1} we used the so-called valuation geometry of $\mathrm{H}(2)^D$, obtained with the help of a computer, to prove that it is not contained in any semi-finite near hexagon as a full isometrically embedded subgeometry, thus giving a partial answer to the famous open problem about existence of semi-finite generalized polygons in this particular case. 
Also, the $G_2(4)$ near octagon $\mG$ was first constructed using the valuation geometry of $\HJ$ (see \cite[Appendix]{ab-bdb:2}). 
The algorithm provided in \cite{ab-bdb:1} to compute the valuation geometry of $\mathrm{H}(2)^D$ can also be used to compute the valuation geometry of $\HJ$ and both these valuation geometries will be crucial to our proofs. 
We have made the computer code, written in GAP \cite{Gap}, required for these computations available online \cite{ab-bdb:code}. 
The valuation geometry of $\mathrm{H}(2)^D$ is constructed in the file \texttt{Suz1.g} while the valuation geometry of $\HJ$ is constructed in \texttt{Suz2.g}. 
This code is also used to prove Lemmas \ref{lemHJ:connected}, \ref{lem:ValBC}, \ref{lem:ValB}, \ref{lem:ValC} and  \ref{lemG:connected}. 

\section{Near polygons and their valuations}
\label{sec:Prelim}

In a partial linear space we can identify each line with the set of points incident with it, and then the incidence relation becomes  set inclusion. We would do so whenever it is convenient. A partial linear space is said to have {\em order} $(s,t)$ if every line is incident with precisely $s+1$ points and if every point is incident with precisely $t+1$ lines. All distances in a partial linear space $\mS=(\m P, \m L, \mI)$ will be measured in its collinearity graph, and denoted by $\dist_{\m S}(\cdot, \cdot)$ or by $\dist(\cdot,\cdot)$ when no confusion could arise. If $x$ is a point of $\mS$, then $\Gamma_i(x)$, for $i \in \mathbb{N}$, will denote the set of points of $\m S$ at distance $i$ from $x$.
Similarly, for a nonempty subset $X \subseteq \m P$, we define 
\[\Gamma_i(X) = \{x \in \m P \, | \,  i = \dist(x, X) := \min_{y \in X} \dist(x, y)\}.\]
Let $X$ be a set of points of $\m S$. Then $X$ is called a \textit{subspace} if for every pair of distinct collinear points in $X$, the line joining those points is contained in $X$.  A subspace $X$ is called \textit{convex}, if for every pair of points $x, y \in X$ all points on a shortest path between $x$ and $y$ are contained in $X$. 

A near $2d$-gon with $d \in \mathbb{N}$ is a partial linear space $\m N = (\m P, \m L, \mI)$ defined by the following axioms:
\begin{enumerate}[(NP1)]
\item the collinearity graph of $\m N$ is a connected graph of diameter $d$;
\item for every $x \in \m P$ and every line $L \in \m L$ there exists a unique point $\pi_L(x)$ incident with $L$ that is nearest to $x$.
\end{enumerate}

It follows that every near $4$-gon is a possibly degenerate generalized quadrangle (see \cite{Pa-Th}). In fact, generalized $2d$-gons with $d \geq 2$ are near $2d$-gons that satisfy the following extra property:
\begin{enumerate}[(GP)]
\item For every pair of points $x, y$ at distance $i \in \{1, 2, \dots, d-1\}$ from each other, the set $\{z \in \m P \, | \, \dist(x, z) = i - 1, \dist(z, y) = 1\}$ is a singleton.
\end{enumerate}
Two lines $L_1$ and $L_2$ of a near polygon are called \textit{parallel} at distance $i$ if $\dist(L_1, L_2) = i$ and for each point $x_1$ on $L_1$, there is a unique point $x_2$ on $L_2$ such that $\dist(x_1, x_2) = i$, or equivalently, if $\dist(L_1, L_2) = i$ and for each point $x_2$ on $L_2$, there is a unique point $x_1$ on $L_1$ such that $\dist(x_1, x_2) = i$.

Let $\m N = (\m P, \m L, \mI)$ be a subgeometry of another near polygon $\m N' = (\m P', \m L', \mI')$, i.e., $\m P \subseteq \m P'$, $\m L \subseteq \m L'$ and $\mI = \mI' \cap (\m P \times \m L)$. 
Then $\mN$ is called a \textit{full} subgeometry of $\mN'$ if for every line $L \in \m L$ we have $\{x \in \m P :  x ~\mI~ L\} = \{x \in \m P' : x ~\mI'~ L \}$ and it is called \textit{isometrically embedded} if for all $x, y \in \m P$ we have $\dist_{\m N}(x, y) = \dist_{\m N'}(x, y)$. 

For every nonempty convex subspace $X$ of a near polygon $\m N = (\m P, \m L, \mI)$, we can define a full isometrically embedded subgeometry of $\m N$ by taking the elements of $X$ as the points of the subgeometry and the lines $L \in \m L$ satisfying $\{x \in X : x ~\mI~  L\} = \{x \in \m P: x ~\mI~ L\}$ as the lines of the subgeometry, with the incidence relation being the one induced by $\mI$. 

The most important class of subgeometries of near polygons are the \textit{quads}. They were introduced by Shult and Yanushka in \cite{Sh-Ya}, and the theory of near polygons with quads was further developed by Brouwer and Wilbrink in \cite{BW}. 
A quad $Q$ of a near polygon $\m N$ is a  set of points that satisfies the following properties. 
\begin{enumerate}[(Q1)]
\item The maximum distance between two points of $Q$ is $2$. 

\item If $x, y \in Q$ are distinct and collinear, then every point incident with the line $xy$ lies in $Q$. 

\item If $x$ and $y$ are two non-collinear points in $Q$, then every common neighbor of $x$ and $y$ is in $Q$. 

\item The subgeometry of $\m N$ determined by those points and lines that are contained in $Q$ is a non-degenerate generalized quadrangle. 
\end{enumerate}
Or succinctly, a quad $Q$ is a convex subspace of $\m N$ that induces a subgeometry isomorphic to a non-degenerate generalized quadrangle. 
Sufficient conditions for existence of quads were given by Shult and Yanushka in \cite[Proposition 2.5]{Sh-Ya}, where they proved that if $a$ and $b$  are two points of a near polygon at distance $2$ from each other, and if $c$ and $d$ are two common neighbors of $a$ and $b$ such that at least one of the lines $ac, ad, bc, bd$ contains at least three points, then $a$ and $b$ are contained in a unique quad.
We will implicitly use this result in our proofs. 

In a near polygon with three points per line, each quad induces a generalized quadrangle of order $(2,t)$, and each such generalized quadrangle is isomorphic to either the $(3 \times 3)$-grid, $W(2)$ or $Q(5, 2)$  (see eg. Section 1.10 in \cite{bdb-book}). We will call the quad a {\em grid-quad}, a {\em $W(2)$-quad} or a {\em $Q(5,2)$-quad} depending on which case occurs. The $(3 \times 3)$-grid is an example of a so-called product near polygon (see Section 1.6 in \cite{bdb-book} for the precise definition) as it can be obtained by taking the direct product $\mathbb{L}_3 \times \mathbb{L}_3$, where $\mathbb{L}_3$ is a line with three points. We will see another example of a product near polygon, $\mathbb{L}_3 \times \mathbb{L}_3 \times \mathbb{L}_3$, in Section \ref{sec:H2}, which is a near hexagon of order $(2, 2)$. 

Let $\mN = (\mP, \mL, \mI)$ be a near $2d$-gon.  A function $f: \mP \rightarrow \mathbb{Z}$ is called a \textit{semi-valuation} of $\mN$ if every line $L$ contains unique point $x_L$ such that $f(x)=f(x_L)+1$ for every point $x$ of $L$ distinct from $x_L$. A {\em valuation} of $\mN$ is  a semi-valuation $f$ for which $\min_{x \in \mP}f(x) = 0$.

Let $f$ be a valuation of $\mN$. Then $M_f$ denotes the maximum value attained by $f$ and $\mathcal{O}_f$ denotes  the set of points with $f$-value $0$. 
It can be easily checked that the set of points of $\mN$ that have $f$-value strictly less than $M_f$ is a {\em hyperplane} $H_f$ of $\mN$, i.e., a proper subset of $\m P$ having the property that each line has either one or all its points in it. 
From the near polygon axioms it also follows that the function $f : \mP \rightarrow \mathbb{Z}$ defined by $f(y) = \dist(x, y)$, where $x$ is a fixed point of $\mN$, is a valuation of $\m N$. 
This is known as the \textit{classical valuation} of $\mN$ with center $x$. 
Two valuations $f_1$ and $f_2$ of $\mN$ are called {\em isomorphic} if there exists an automorphism $\theta$ of $\mN$ such that $f_2 = f_1 \circ \theta$. 
Thus, all classical valuations of $\mN$ are isomorphic if $\mathrm{Aut}(\mN)$ acts transitively on the points of $\mN$, which is true for all the Suzuki tower near polygons. 

Two valuations $f_1$ and $f_2$ of $\mN$ are called \textit{neighboring valuations} if there exists an $\epsilon \in \mathbb{Z}$  such that $|f_1(x) - f_2(x) + \epsilon| \leq 1$ for every point $x$ of $\mathcal{N}$. The number $\epsilon$ (necessarily belonging to $\{ -1,0,1 \}$) is uniquely determined, except when $f_1=f_2$, in which case there are three possible values for $\epsilon$, namely $-1$, $0$ and $1$. 
Suppose $\mN$ is  a near polygon in which every line has precisely three points. 
Let $f_1, f_2$ be two neighboring valuations and let $\epsilon \in \{-1, 0, 1\}$ be such that $|f_1(x) - f_2(x) + \epsilon| \leq 1$ for all points $x$. 
For points $x$ satisfying $f_1(x) = f_2(x) - \epsilon$ we define $f_3'(x) = f_1(x) - 1 = f_2(x)  - \epsilon - 1$ and for other points $x$ we define $f_3'(x) = \max \{f_1(x), f_2(x) - \epsilon\}$. 
Let $m$ be the minimal value attained by $f_3'$. 
Then it is easily seen that $f_3'$ is a semi-valuation and thus the function defined by $f_3(x) =  f_3'(x) - m$ is a valuation of $\m N$, which we denote by $ f_1 \ast f_2$. If $f_1=f_2$, then we have $f_3=f_1=f_2$, regardless of the value of $\epsilon$. For neighboring valuations $f_1, f_2$ of $\m N$ and $f_3 = f_1 \ast f_2$ we have
(i) $f_2 \ast f_1 = f_1 \ast f_2 = f_3$; (ii) $f_1$ and $f_3$ are neighboring valuations and $f_1 \ast f_3 = f_2$; (iii) $f_2$ and $f_3$ are neighboring valuations and $f_2 \ast f_3 = f_1$. 
For more on the basic theory of valuations we refer to \cite[Section 2]{ab-bdb:1} and \cite{bdb:Gn}. 
The following result establishes the main connection between valuations and isometric embeddings of near polygons. 
\begin{lem}
\label{lem:linetypes}
\label{lem:embeddings}
Let $\mN = (\mP, \mL, \mI)$ be a near polygon which is an isometrically embedded full subgeometry of a near polygon $\mN' = (\mP', \mL', \mI')$. 
Then the following holds: 
\begin{enumerate}
\item[$(1)$] For every point $x$ in $\mP'$ the function $f_x: \mP \rightarrow \mathbb{N}$ defined by $f_x(y) := d(x,y) - d(x, \mP)$ is a valuation of $\mN$.
\item[$(2)$] For every pair of distinct collinear points $x_1$ and $x_2$ in $\mN'$, the valuations $f_{x_1}$ and $f_{x_2}$ are neighboring.
\item[$(3)$] Say every line of $\mN'$ is incident with three points and let $\{ x_1,x_2,x_3 \}$ be a line of $\mN'$. Then $f_{x_1} \ast f_{x_2} = f_{x_3}$. In particular, if two of $f_{x_1},f_{x_2},f_{x_3}$ coincide then they are all equal.
\end{enumerate}
\end{lem}
\begin{proof}
See \cite[Lemma 2.2]{ab-bdb:1}. 
\end{proof}

\begin{lem} \label{new}
Let $\mN = (\mP, \mL, \mI)$ be a near polygon which is an isometrically embedded full subgeometry of a near $2d$-gon $\mN' = (\mP', \mL', \mI')$, and for every point $x$ of $\mN'$, let $f_x$ be the valuation of $\mN$ as defined in Lemma \ref{lem:linetypes}. Then:
\begin{enumerate}
\item[$(1)$] If $x$ is a point of $\mN'$ such that $\dist(x,\mP)=i$, then $M_{f_x} \leq d - i$. 
\item[$(2)$] If $x$ is a point at distance $1$ from $\mN$ such that $|\mO_{f_x}| = 1$, then there is a unique point $\pi_{\mN}(x)$ in $\mN$ collinear with $x$.
\end{enumerate}
\end{lem}
\begin{proof}
This immediately follows from the definition of the map $f_x$.
\end{proof}

Let $\m N$ be a near polygon that has three points on each line and let $V$ be the set of valuations of $\m N$. 
The \textit{valuation geometry} of $\m N$ is the partial linear space $\m V$ defined by taking the set $V$ as points and the triples $\{f_1, f_2, f_3\}$ of pairwise distinct and neighboring valuations that satisfy $f_1 \ast f_2 = f_3$ as lines. 
We observe that $\mathrm{Aut}(\m N)$ acts on the valuation geometry $\m V$ by the map $(f, \theta) \in V \times \mathrm{Aut}(\m N) \mapsto f \circ \theta^{-1}$ (and thus $f(x) =  (f \circ \theta^{-1}) (\theta(x))$, $\forall x \in \m P ~ \forall f \in V ~ \forall \theta \in \mathrm{Aut}(\m N)$). 
When computing valuations we will only record the information about different orbits under this action by giving each orbit a different label (see eg. Table \ref{tab1:HD2}) and noting the essential properties of the valuations in that orbit. 
Similarly, the lines of $\m V$ will be given a type, which is just a sorted tuple of the type of points on that line, and we record the information about the number of $\mV$-lines of a given type incident to a fixed valuation of a given type in a separate table (see eg. Table \ref{tab2:HD2}). 
Say $\m N$ is a full isometrically embedded subgeometry of another near polygon $\m N'$. 
Then by Lemma \ref{lem:embeddings} the points and lines of $\m N'$ induce points and lines of the valuation geometry $\m V$ of $\m N$. 
We define the type of a point or line in $\m N'$ to be the type of the corresponding point or line of $\m V$. 
Note that the points/lines of $\m N'$ of the same type are not necessarily isomorphic under the action of $\mathrm{Aut}(\m N')$. 
Lines of the valuation geometry $\mV$ of a fixed near polygon $\mN$ will be referred to as $\mV$-lines and the valuation of $\m N$ induced by a point $x$ of $\m N'$ (see Lemma \ref{lem:embeddings}) will be denoted by $f_x$. 

\section{Near hexagons of order $(2,2)$}
\label{sec:H2}

In this section we classify all near hexagons $\mN$ of order $(2, 2)$.  
We know that if every pair of points in $\mN$ at distance $2$ from each other have a unique common neighbor, then $\mN$ is a generalized hexagon. 
In that case we can use the result of Cohen and Tits \cite{Co-Ti} to say that $\mN$ is either isomorphic to the split Cayley hexagon $\mathrm{H}(2)$ or its dual $\mathrm{H}(2)^D$. 
Similarly, if every pair of points in $\m N$ at distance $2$ have more than one common neighbor then we can use \cite[Theorem 1.1]{BCHW} to conclude that $\m N$ must be isomorphic to $\mathbb{L}_3 \times \mathbb{L}_3 \times \mathbb{L}_3$ (number  (xi) in their classification). 
We will prove that these are the only possible cases. 
First we need a basic result on near polygons with an order. 

\begin{lem}[{\cite[Theorem 1.2]{bdb-book}}]
\label{lem:count}
Let $\m N = (\m P, \m L, I)$ be a finite near $2d$-gon, $d \geq 1$, of order $(s,t)$ and let $x$ be a point of $\m N$. 
Then 
\[\sum_{y \in \m P} \left( \frac{-1}{s}\right)^{\mathrm{d}(x,y)} = 0.\]
\end{lem}
\begin{proof}
By (NP2), for every line $L$ the sum $\sum_{y \mI L} (\frac{-1}{s})^{\dist(x,y)}$ is $0$. 
Therefore, we have 
\[0 = \sum_{L \in \m L} \sum_{y \mI L} \left(\frac{-1}{s}\right)^{\dist(x,y)} =  \sum_{y \in \m P} \sum_{L \mI y} \left(\frac{-1}{s}\right)^{\dist(x,y)} = (t+1) \sum_{y \in \m P} \left(\frac{-1}{s}\right)^{\dist(x,y)} .\]
\end{proof}

\begin{lem}
\label{lem:quad_order}
Let $\m N$ be a finite near hexagon of order $(s, t)$ and $Q$ a quad of $\m N$ that has order $(s, t')$. 
Then $t' < t$. 
\end{lem}
\begin{proof}
We know that $t' \leq t$. 
For the sake of contradiction, assume that $t' = t$. 
Let $x$ be a point of $Q$. 
Since all lines of $\m N$ through $x$ are already contained in $Q$, $x$ cannot be collinear with any point that is not contained in $Q$. 
But then, there cannot be any points of $\m N$ that lie outside $Q$, as the collinearity graph of $\m N$ is connected. 
Thus $\m N = Q$, which is a contradiction. 
\end{proof}

\begin{lem}
Let $\m N$ be a near hexagon of order $(2, 2)$. Then the number of common neighbors of a pair of points at distance $2$ from each other is a constant $c \in \{1, 2\}$. 
\end{lem}
\begin{proof}
Let $v$ denote the total number of points of $\m N = (\m P, \m L, \mI)$. 
For a fixed point $x$ let $n_i(x)$ denote the number of points at distance $i \in \{ 0,1,2,3 \}$ from $x$. 
For all $x \in \m P$, we have $n_0(x) = 1$, $n_1(x) = 6$, and thus
\begin{equation} \label{eq:1} 
n_2(x) + n_3(x) = v - 7. 
\end{equation}
\\ 
By Lemma \ref{lem:count} we have 
\begin{equation}
\label{eq:2}
n_0(x) - \frac{n_1(x)}{2} + \frac{n_2(x)}{4} - \frac{n_3(x)}{8} = 0.
\end{equation}
\\
Solving equations (\ref{eq:1}) and (\ref{eq:2}) we get that $n_2(x) = (v + 9)/3$ and $n_3(x) = (2v - 30)/3$ for all $x \in \m P$.
Therefore these numbers only depend on $v$, and we can define constants 
\begin{equation}
n_0 = 1, n_1 = 6, n_2 = (v + 9)/3, n_3 = (2v - 30)/3.
\end{equation}
\\ 
By Lemma \ref{lem:quad_order} all quads of $\m N$ are grid-quads ($\cong \mathbb{L}_3 \times \mathbb{L}_3$). 
For a point $x$, let $N(x)$ be the number of grid-quads that contain $x$.  
Then the number of points at distance $2$ from $x$ that are contained in a grid along with $x$ is equal to $4N(x)$ since there is a unique quad through a pair of points at distance $2$ which have more than one common neighbor. 
Double counting edges between $\Gamma_1(x)$ and $\Gamma_2(x)$ we get that $2 \cdot 4N(x) + 1 \cdot (n_2 - 4N(x)) = n_1 \cdot 4$, and hence $N(x) = (63 - v)/12$.
So, the total number of grid-quads through a point is a constant given by $N := (63 - v)/12$. 

It is known that there cannot be more than one quad through a pair of intersecting lines (see eg. \cite[Thm. 1.4]{bdb-book}).
Since the number of lines through each point is $3$, we must have $N \in \{0, 1, 2, 3\}$. 
Since $v = 63 - 12N$, using double counting we get that the total number of grid-quads in $\m N$ is \[N(63 - 12N)/9.\]
This number is not an integer if $N \in \{1, 2\}$. 
Therefore, $N$ must be $0$ or $3$. 
If $N$ is $0$, then there are no quads, and hence every two points at distance $2$ from each other have a unique common neighbor. 
Say $N$ is equal to $3$ and let $x, y$ be a pair of points at distance $2$ from each other. 
Every line through $x$ is contained in precisely two of the three grid-quads through $x$. 
Thus for every neighbor $z$ of $x$, the two lines through $z$ that contain a point of $\Gamma_2(x)$ are contained in grids through $x$. 
This implies that there is a grid through $x$ and $y$, i.e., the number of common neighbors between $x$ and $y$ is $2$. 
\end{proof}

\begin{cor}
Every near hexagon of order $(2,2)$ is isomorphic to one of the following: $\mathrm{H}(2)$, $\mathrm{H}(2)^D$, $\mathbb{L}_3 \times \mathbb{L}_3 \times \mathbb{L}_3$. 
\end{cor}

It is known that the generalized hexagon $\mathrm{H}(2)^D$ has subgeometries isomorphic to $\mathrm{H}(2,1)$, while $\mathrm{H}(2)$ does not have such subgeometries. We can thus finish the proof of Theorem \ref{thm:H2} by showing that also $\mathbb{L}_3 \times \mathbb{L}_3 \times \mathbb{L}_3$ does not contain any subgeometry isomorphic to $\mathrm{H}(2, 1)$. So, let $\mH_1 \cong \mathrm{H}(2,1)$ be a subgeometry of $\mH_2 \cong \mathbb{L}_3 \times \mathbb{L}_3 \times \mathbb{L}_3$. We know that $\mH_1$ has $21$ points and $\mH_2$ has $27$ points. 
Each point of $\mH_1$ is collinear with exactly two points of $\mH_2 \setminus \mH_1$. Since $|\mH_2 \setminus \mH_1| = 6$, there must be a point in $\mH_2 \setminus \mH_1$ collinear with at least $(21 \times 2)/6 = 7$ points of $\mH_1$. This contradicts the fact that $\mH_2$ has order $(2,2)$. 

\section{Characterization of the Hall-Janko near octagon}
\label{sec:HJ}

For this section let $\mN$ be a near octagon of order $(2,4)$ with a generalized hexagon $\mH$ isomorphic to $\mathrm{H}(2)^D$ isometrically embedded in it. 
The valuation geometry $\mV$ of $\m H$ is described in Tables \ref{tab1:HD2} and \ref{tab2:HD2}.
In Table \ref{tab1:HD2}, the column Value Distribution denotes the distribution of points of $\m H$ as per the value they have. 
In Table \ref{tab2:HD2}, the entries denote the number of lines of a given type through a point of a given type in $\mV$.  
From Table \ref{tab2:HD2} it can be seen that the set of valuations of type $A$ and $B$ form a subspace of $\mV$, and hence we can define a full subgeometry $\mV_{A,B}$ of $\mV$ induced by these valuations. 
In this section we will show that $\mN$ is isomorphic to $\mV_{A, B}$. 
Since $\HJ$ is a near octagon of order $(2,4)$ with $\mathrm{H}(2)^D$ isometrically embedded in it, this will prove Theorem \ref{thm:HJ}. 

\begin{table}[!htbp]
\begin{center}
\begin{tabular}{|c||c|c|c|c|c|}
 \hline
  Type & $\#$ & $M_f$ & $|\mO_f|$ & $|H_f|$  & Value distribution \\ \hline \hline
  $A$ & $63$ & $3$ & $1$ & $31$ & $[1, 6, 24, 32]$  \\ \hline
  $B$ & $252$ & $3$ & $1$ & $47$ & $[1, 14, 32, 16]$ \\ \hline
  $C$ & $252$ & $2$ & $1$ & $23$ & $[ 1, 22, 40, 0 ]$ \\ \hline
  $D$ & $1008$ & $2$ & $5$ & $31$ & $[5, 26, 32, 0]$ \\ \hline
\end{tabular}
\end{center}
\caption{The valuations of $\mathrm{H}(2)^D$}
\label{tab1:HD2}
\end{table}

\begin{table}[!htbp]
\begin{center}
\begin{tabular}{|c||c|c|c|c|}
\hline 
Type & $A$ & $B$ & $C$ & $D$ \\ \hline \hline

$AAA$ & $3$ & -- & -- & -- \\ \hline

$ABB$ & $2$ & $1$ & -- & -- \\ \hline 

$ACC$ & $2$ & -- & 1 & -- \\ \hline 

$ADD$ & $24$ & -- & -- & $3$ \\ \hline 

$BBB$ & -- & $4$ & -- & -- \\ \hline

$BCC$ & -- & $1$ & $2$ & -- \\ \hline 

$BDD$ & -- & $4$ & -- & $2$ \\ \hline 

$CCC$ & -- & -- & 8 & -- \\ \hline 

$CCD$ & -- & -- & $40$ & $5$ \\ \hline

$CDD$ & -- & -- & $4$ & $2$ \\ \hline

$DDD$ & -- & -- & -- & $10$ \\ \hline   
\end{tabular}
\end{center}
\caption{The lines of the valuation geometry $\mathcal{V}$ of $\mathrm{H}(2)^D$} 
\label{tab2:HD2}
\end{table}

\begin{lem} \label{lemHJ:basic}
\begin{enumerate}[$(a)$]
\item Each point of $\mN$ is at distance at most $2$ from $\mH$. 
\item Points at distance $1$ from $\mH$ must be of type $A$, $B$ or $C$.
\item Points at distance $2$ from $\mH$ must be of type $C$ or $D$. 
\end{enumerate}
\end{lem}
\begin{proof}
Since the maximum value of a valuation of $\mH$ is at least $2$ (see Table \ref{tab1:HD2}) and the diameter of $\mN$ is $4$, by Lemma \ref{new}(1) the distance of any point of $\mN$ to $\mH$ is at most $2$. This proves ($a$). Again by Lemma \ref{new}(1), if a point $x$ of $\mN$ lies at distance 2 from $\mH$, then $f_x$ has maximum value at most $2$, implying that $x$ can only be of type $C$ or $D$. 

Now, let $x$ be a point of type $D$ at distance $1$ from $\mH$. The five points with $f_x$-value $0$ in $\mH$ must be collinear with $x$ and necessarily be of type $A$ (all points in $\mH$ are of type $A$).
By Lemma \ref{lem:linetypes} this gives rise to five distinct $\mV$-lines of type $ADD$ through a valuation of type $D$ in the valuation geometry $\mV$ of $\mH$, which contradicts the corresponding entry in Table \ref{tab2:HD2}. 
\end{proof}

\begin{lem} \label{lemHJ:dist1}
Each point of $\mN$ at distance $1$ from $\mH$ must be of type $A$ or $B$. 
\end{lem}
\begin{proof}
Let $x$ be a point of type $C$ at distance $1$ from $\mH$. 
By Lemma \ref{new}(2) and Table \ref{tab1:HD2}, there is a unique point $x'$ in $\mH$ collinear with $x$. 
Again from Table \ref{tab1:HD2} we see that there are $22$ points with $f_x$-value $1$ in $\mH$, which must necessarily be at distance $2$ from $x$. 
Six of these points are neighbors of $x'$ in $\mH$ and these are the only ones that have a common neighbor with $x$ that lies inside $\mH$ (namely $x'$). 
The remaining $16$ points give rise to neighbors of $x$ that lie outside $\mH$. 
Since the order of $\mN$ is $(2, 4)$ and $x'$ lies in $\mH$ there are only $9$ neighbors of $x$ that lie outside $\mH$. 
Therefore, at least one such neighbor $y$ must be collinear with more than one point in $\mH$. 
By Lemma \ref{lemHJ:basic} the point $y$ must be of type $A$, $B$ or $C$ and for each of these possibilities we have $|\mO_{f_y}| = 1$. 
This contradicts Lemma \ref{new}(2). 
\end{proof}

\begin{lem} \label{lemHJ:AB}
If $x$, $y$ are two points of $\mN$, not contained in $\mH$, of type $A$ and $B$ respectively, then $x$ and $y$ cannot be collinear.
\end{lem}
\begin{proof}
Let $x$, $y$ be such points and suppose they are collinear. 
By Lemma \ref{lemHJ:basic}  they must be at distance $1$ from $\mH$. 
The three valuations induced by the three points on the line $xy$ must be distinct (see Lemma \ref{lem:linetypes}) and therefore, the line $xy$ gives rise to a $\mV$-line of $\mH \cong \mathrm{H}(2)^D$. 
From Table \ref{tab2:HD2} it follows that the line $xy$ is of type $ABB$.
Let $y'$ be the unique neighbor of $y$ in $\mH$. 
Then by a similar reasoning the line $yy'$ is also of type $ABB$. 
But, in the valuation geometry there is a unique line of type $ABB$ through a valuation of type $B$.
Therefore, the points $x$ and $y'$ induce the same type $A$ valuation, which shows that $\mO_{f_x} = \{y'\}$ and hence, $x$ and $y'$ are collinear. 
This contradicts (NP2). 
\end{proof}

\begin{lem} \label{lemHJ:B}
If $x$ is a point of type $B$ in $\mN$, then it has a unique neighbor in $\mH$ and all the other neighbors of $x$ must induce distinct type $B$ valuations of $\mH$. 
\end{lem}
\begin{proof}
Let $x$ be such a point, necessarily at distance $1$ from $\mH$ by Lemma \ref{lemHJ:basic}. 
By Lemma \ref{new}(2) and Table \ref{tab1:HD2}, it has a unique neighbor, say $x'$, in $\mH$. 
There are $14$ points with $f_x$-value $1$ in $\mH$ and $6$ of them are the neighbors of $x'$. 
The remaining $8$ must give rise to neighbors of $x$ lying outside $\mH$. 
Let $y$ be such a neighbor. 
By Lemmas \ref{lemHJ:dist1} and \ref{lemHJ:AB}, $y$ must be of type $B$ and then by Lemma \ref{new}(2) it cannot lie on the line $xx'$. 
Therefore, we get $8$ type $B$ neighbors of $x$ each corresponding to a distinct valuation of $\mH$ (since the set $\mO_f$ is distinct for each such valuation).
The third point on the line $xx'$ must also be of type $B$ and induce a valuation distinct from all other type $B$ neighbors of $x$. 
Since the order of $\mN$ is $(2,4)$, we have accounted for all neighbors of $x$. 
\end{proof}

\noindent The following is an immediate consequence of Lemma \ref{lemHJ:B}.

\begin{cor}\label{corHJ:B}
There are no lines in $\mN$ of type $BCC$ or $BDD$. 
\end{cor}

\begin{lem}\label{lemHJ:noCD}
There is no point in $\mN$ of type $C$ or $D$. 
\end{lem}
\begin{proof}
Let $x$ be such a point, necessarily at distance $2$ from $\mH$ (see Lemmas \ref{lemHJ:basic} and \ref{lemHJ:dist1}). 
We treat the two cases separately.
\\
\textit{Case 1}: Let $x$ be of type $D$. 
By Table \ref{tab1:HD2}, $|\mH \cap \Gamma_2(x)| = 5$. 
Every line through $x$ that contains a point in $\Gamma_1(\mH)$ must be of type $ADD$ by Table \ref{tab2:HD2}, Lemma \ref{lemHJ:dist1} and Corollary \ref{corHJ:B}. 
Since each of these lines has exactly one point which lies in $\Gamma_1(\mH)$, and since that point (of type $A$) has a unique neighbor in $\mH$, it must be the case that all five lines through $x$ are of type $ADD$.
In fact, this also shows that these five lines correspond to five distinct lines in the valuation geometry. 
But we know from Table \ref{tab2:HD2} that there are only three lines of type $ADD$ through a point of type $D$ in the valuation geometry, a contradiction.
\\
\textit{Case 2}: Let $x$ be of type $C$. 
Since $|\mO_{f_x}| = 1$, we have $\mO_{f_x} = \Gamma_2(x) \cap \mH = \{x'\}$ for some $x' \in \mH$. 
Since there are no points of type $D$ (by \textit{Case 1}) and no lines of type $BCC$ (by Corollary \ref{corHJ:B}), all lines through $x$ must be of type $ACC$ or $CCC$ by Table \ref{tab2:HD2}. 
Each of the type $A$ neighbors of $x$ induces the valuation $f_{x'}$ of $\mH$. 
Therefore, besides the $6$ neighbors of $x'$ in $\mH$, every point of $\mH$ that has $f_x$-value $1$ must be at distance $2$ from a type $C$ neighbor of $x$. 
There are $16 = 22 - 6$ points with $f_x$-value $1$ in $\mH$ that are not neighbor of $x'$. 
Since there are at most $9$ type $C$ neighbors of $x$, there must be a type $C$ neighbor $y$ of $x$ which is at distance $2$ from two distinct points of $\mH$.
This contradicts the fact that $|\mO_{f_y}| = 1$ ($y$ is a type $C$ point).
\end{proof}

\medskip \noindent By Lemmas \ref{new}(2), \ref{lemHJ:basic} and \ref{lemHJ:noCD}, we have:

\begin{cor} \label{new2}
Every point $x$ of $\mN$ not contained in $\mH$ has type A or B, and lies at distance 1 from a unique point of $\mH$ (the projection of $x$ in $\mH$).
\end{cor}

\begin{lem}
\label{lemHJ:QuadIntersection}
Let $Q$ be a quad of $\mN$ that intersects $\mH$ nontrivially. Then $Q \cap \mH$ is either a singleton or a line.
\end{lem}
\begin{proof}
Say $Q \cap \mH$ is not a singleton.
Since $Q \cap \mH$ is a subspace, it suffices to show that there are no two non-collinear points in $Q \cap \mH$. 
Let $x$, $y$ be two non-collinear points in $Q \cap \mH$.
Since $Q$ is a non-degenerate generalized quadrangle, there are at least two common neighbors of $x$ and $y$ in $Q$. 
Since points at distance $2$ in $\mH \cong \mathrm{H}(2)^D$ have a unique common neighbor and $Q \cap \mH$ is a convex subspace, at least one of these common neighbors must lie outside $\mH$. 
This gives rise to a point at distance $1$ from $\mH$ with two neighbors ($x$ and $y$) in $\mH$, which contradicts Corollary \ref{new2}.
\end{proof}

\begin{lem}\label{lemHJ:W2}
If $x$ is a point of type $A$ in $\mN$ which is not contained in $\mH$, then there exists a unique $W(2)$-quad $Q$ containing $x$ and its projection $x'$ in $\mH$.
For this quad $Q$ we have:
\begin{enumerate}[$(a)$]
\item $Q$ intersects $\mH$ in a line $M$.
\item Let $M_1$ and $M_2$ be the two lines through $x'$ in $\mH$ other than $M$. Then the two lines through $x$ that are not contained in $Q$ can be labeled $L_1$ and $L_2$ such that $L_1, M_1$ are parallel and at distance $1$ from each other, and $L_2, M_2$ are parallel and at distance $1$ from each other. 
\end{enumerate}
\end{lem}
\begin{proof}
Let $x$ be a point of type $A$ outside $\mH$ and let $x'$ be the unique point in $\mH$ collinear with $x$. 
We have $f_x = f_{x'}$. 
Each of the four lines through $x$ which lie outside $\mH$ are of type $AAA$ by Lemmas \ref{lemHJ:AB} and \ref{lemHJ:noCD}. 
From Table~\ref{tab2:HD2} we see that there are only three distinct lines of type $AAA$ through a valuation of type $A$ in the valuation geometry. 
Therefore, there exists two lines $K_1$, $K_2$ through $x$ which lie in $\Gamma_1(\mH)$ and induce the same set of valuations on $\mH$. 
Let $\{f_1, f_2, f_3\}$ be this set with $f_1 = f_x = f_{x'}$.
Since all three lines through $x'$ which lie inside $\mH$ correspond to distinct lines of type $AAA$ in the valuation geometry, at least one of them, say $M$, must induce the set $\{f_1, f_2, f_3\}$ of valuations.

Let $y \neq x'$ be a point on $M$. Then $y$ is collinear with a point on $K_1$ and a point on $K_2$ both of which induce the valuation equal to $f_y$. Therefore, $x$, $y$ are two points at distance $2$ in a near polygon with at least three common neighbors. From the existence result of quads, it follows that $x$ and $y$ lie in a unique quad $Q$. By the classification of quads of order $(2,t)$ and Lemma \ref{lem:quad_order}, $Q$ must be isomorphic to $W(2)$, the unique generalized quadrangle of order $(2,2)$. 
From Lemma \ref{lemHJ:QuadIntersection} it follows that $Q \cap \mH = M$ and hence, for each point on $M$, the two lines through it going out of $\mH$ are contained in $Q$. 
Hence none of the points on $M$ can be contained in a $W(2)$-quad other than $Q$.

Now let $L$ be a line through $x$ not contained in $Q$. 
$L$ necessarily induces a set of valuations other then $\{f_1, f_2, f_3\}$. 
There are only two other possibilities and both of them are induced by lines through $x'$ contained in $\mH$, but not in $Q$. 
Therefore there must be a line $L'$ through $x'$ inducing the same set of valuations as $L$. The correspondence $L \mapsto L'$ between the set of lines through $x$ not contained in $Q$ and the set of lines through $x'$ in $\mH$ distinct from $M$ is a bijection as otherwise there would exist another $W(2)$-quad through the line $xx'$ but we have already proved that there is a unique such quad. 
\end{proof}

\begin{lem}
There is no point in $\mN$ of type $A$ outside $\mH$. 
\end{lem}
\begin{proof}
Let $x$ be a point of type $A$ outside $\mH$. By Lemma \ref{lemHJ:W2} it lies in a unique $W(2)$-quad $Q$ which intersects $\mH$ in a line $L$. By 
Lemma~\ref{lemHJ:AB} and Corollary~\ref{new2}, all points of $Q \setminus L$ have type A. Let $x'$ be the projection of $x$ in $\mH$ and $y'$ a neighbor of $x'$ in $\mH$ lying on a line through $x'$ other than $L$. By Lemma~\ref{lemHJ:W2}, there exists a unique neighbor $y$ of $y'$ outside $\mH$ and collinear with $x$. Again by Lemma \ref{lemHJ:AB} and Corollary~\ref{new2}, the point $y$ has type A. So, by Lemma~\ref{lemHJ:W2}, there exists a unique $W(2)$-quad $S$ containing $y$ and $y'$. The $W(2)$-quads $Q$ and $S$ are disjoint.

Suppose $p$ is a neighbor of $x'$ contained in $Q \setminus L$. As $p$ has type A, there exists by Lemma~\ref{lemHJ:W2} a unique line through $p$ disjoint from $\mH$ that is parallel and at distance 1 from the line $x'y'$, implying that there is a common neighbor of $p$ and $y'$ in $S \setminus \mH$. This implies that we can label the two lines of $Q$ through $x'$ distinct from $L = Q \cap \mH$ by $T_1$ and $T_2$ and the two lines of $S$ through $y'$ distinct from $S \cap \mH$ by $U_1$ and $U_2$ such that $T_1,U_1$ are parallel and at distance 1 from each other, and $T_2,U_2$ are parallel and at distance 1 from each other. 

Now, consider a point $z$ in $S$ which is not collinear with $y'$. If $z$ is at distance $1$ from $Q$ (necessarily from a point of type A of $Q \setminus L$), then by Lemma \ref{lemHJ:W2} its projection $z' \in S \cap \mH$ is collinear with a point on the line $L$ in $\mH$, contradicting the fact that $\mH$ is a generalized hexagon.
So $z$ (as well as every point of $S$ non-collinear with $y'$) must be at distance at least $2$ from $Q$. The two lines of $S$ through $y'$ distinct from $S \cap \mH$ are parallel and at distance 1 from a line of $Q$. So, taking the projection of $z$ on these two lines, we see that there are two points in $Q$ at distance $2$ from $z$. It is known that the point $z$ induces a classical or an ovoidal valuation of $Q$ (see e.g. \cite[Thm. 1.22]{bdb-book}). Since there are two points in $Q$ at distance $2$ from $z$, the point $z$ must induce an ovoidal valuation of $Q$. Since there are five points in an ovoid in $Q$ ($\cong W(2)$), each of the five lines through $z$ must contain a (necessarily unique) point at distance $1$ from $Q$. Thus the projection of $z$ on $\mH$, $z'$, must be collinear with a point in $Q$. Now, all the lines through $z'$ are either in $\mH$ or in $S$. Since $S$ and $Q$ are disjoint and $\mH$ is a generalized hexagon, we have a contradiction.
\end{proof}

\noindent
Therefore, $\mN$ has the following description:
\begin{itemize}
\item each point of $\mN$ is at distance at most $1$ from $\mH$;

\item each point of $\mN$ that lies in $\mH$ induces a valuation of type $A$, and

\item each point of $\mN$ that does not lie in $\mH$ induces a valuation of type $B$. 
\end{itemize}
Obviously, distinct points of $\mN$ induce distinct type A valuations. To prove that $\mN$ is isomorphic to $\mV_{A, B}$, we first show that no two points in $\mN$ can induce the same type $B$ valuation of $\mH$, which will give us a bijection between the point sets of these geometries.

\begin{lem}\label{lemHJ:B1}
If $g_1$ and $g_2$ are two distinct valuations of type $B$ collinear to each other in the valuation geometry, then we have $|\{x \in \mN \setminus \mH : f_x = g_1\}| = |\{x \in \mN \setminus \mH: f_x = g_2 \}|$
\end{lem}
\begin{proof}
Let $g_1$, $g_2$ be two such valuations of type $B$. Say a point $y$ in $\mN$ induces the valuation $g_1$ of $\mH$. If $g_1$ and $g_2$ lie on $\mV$-line of type $ABB$, then the third point on the line joining $y$ and the unique neighbor of $y$ that lies in $\mH$ induces the valuation $g_2$, giving us a bijection between $\{x \in \mN \setminus \mH : f_x = g_1\}$ and $\{x \in \mN \setminus \mH: f_x = g_2 \}$.
So, say $g_1$ and $g_2$ lie on $\mV$-line of type $BBB$.
Then by Lemma~\ref{lemHJ:B} we know that each of the four lines through $y$ which do not intersect $\mH$ must induce distinct $\mV$-lines of type $BBB$. 
But we know from Table~\ref{tab2:HD2} that there are exactly four such $\mV$-lines containing $g_1$, and hence $g_2$ is contained in exactly one of them. 
Therefore, there must be precisely one neighbor of $y$ in $\mN \setminus \mH$ which induces the valuation $g_2$. 
This gives a bijection between $\{x \in \mN \setminus \mH: f_x = g_1\}$ and $\{x \in \mN \setminus \mH: f_x = g_2\}$. 
\end{proof}

\begin{lem}
\label{lemHJ:B2}
\label{lemHJ:connected}
The subgeometry of $\mV$ defined on the type $B$ valuations by the lines of type $ABB$ and $BBB$ is connected.
\end{lem}
\begin{proof}
This is checked by computer computation, see \cite{ab-bdb:code}.
\end{proof}

\begin{cor}\label{corHJ:BOnce}
For each type $B$ valuation $f$ of $\mH$, there exists exactly one point $x \in \mN$ with $f_x = f$.
\end{cor}
\begin{proof}
Since each point of $\m H$ is collinear with precisely four points of $\mN \setminus \mH$, and each point of $\mN \setminus \mH$ has a unique neighbor in $\mH$, we have $|\mN \setminus \mH| = 4 \times |\mH| = 252$.
By Lemmas \ref{lemHJ:B1} and \ref{lemHJ:B2} we know that for every pair of type B valuations there exist equally many points in $\mN$ which induce those valuations. But from Table~\ref{tab1:HD2} we can see that there are exactly $252$ valuations of type $B$. Therefore each type $B$ valuation is induced exactly once. 
\end{proof}

Now we can prove that $\mN$ is isomorphic to $\mV_{A,B}$ as follows. Map every point $x$ of $\mN$ to the valuation of type $T \in \{A, B\}$ that it induces. Since no two points of $\mN$ induce the same valuation (see Corollary \ref{corHJ:BOnce}), for every line $L = \{x, y, z\}$ of $\mN$ the triple $\{f_x, f_y, f_z\}$ is a $\mV$-line. Map every line of $\mN$ to this corresponding line of $\mV_{A, B}$. Since $\mN$ and $\mV_{A,B}$ have the same number of points and the same order $(2,4)$ the above maps between the point and line sets of $\mN$ and $\mV_{A,B}$ are bijections and define an isomorphism between the two geometries. Thus, every near polygon of order $(2,4)$ that contains an isometrically embedded generalized hexagon $\mH$ isomorphic to $\mathrm{H}(2)^D$ must be isomorphic to $\mV_{A, B}$, which proves Theorem \ref{thm:HJ}.

\section{Characterization of  the $G_2(4)$ near octagon}
\label{sec:G24}

For this section let $\mN$ be a near octagon with three points on each line containing a suboctagon $\mH$ isomorphic to $\HJ$ isometrically embedded in it. 
The valuation geometry $\mV$ of $\mH \cong \HJ$ is given in Tables \ref{tab1:HJ} and \ref{tab2:HJ}. 
The main purpose of this section is to show that if $\mN$ has order $(2,10)$, then $\mN$ is isomorphic to the $G_2(4)$ near octagon. 
In \cite[Appendix]{ab-bdb:2} it was shown that the $G_2(4)$ near octagon can be constructed by taking the valuations of type $A$, $B$ and $C$ as points and the $\mV$-lines of type $AAA$, $ABB$, $ACC$, $BBC$ and $CCC$ as lines. 
Therefore, we will show that if $\mN$ has order $(2,10)$, then it consists of points of type $A$, $B$ or $C$ and lines of type $AAA, ABB, ACC, BBC$ or $CCC$, with each type occurring exactly once. 
First we derive some general results that are true for any near octagon $\mN$ with three points on each line that contains $\mH$ as a full isometrically embedded subgeometry and later restrict ourselves to the case when $\mN$ has order $(2,10)$. 
Lemmas \ref{lem:ValBC} to \ref{lemG:connected} are proved using the computer model of the valuation geometry that we have constructed, see \cite{ab-bdb:code}.
\begin{table}[!htbp]
\begin{center}
\begin{tabular}{|c||c|c|c|c|c|}
\hline
Type & \# & $M_f$ & $|\mathcal{O}_f|$ & value distribution \\
\hline
\hline
$A$ & $315$ & $4$ & $1$ & $[1,10, 80, 160, 64]$ \\
\hline 
$B$ & $630$ & $3$ & $1$ & $[1, 10, 112, 192, 0]$ \\
\hline 
$C$ & $3150$ & $3$ & $1$ & $[1, 26, 128, 160, 0]$ \\
\hline 
$D$ & $1008$ & $2$ & $5$ & $[5, 110, 200, 0, 0]$ \\
\hline 
$E$ & $2016$ & $2$ & $25$ & $[25, 130, 160, 0, 0]$ \\
\hline 
\end{tabular}
\end{center}
\caption{The valuations of Hall-Janko near octagon $\HJ$}
\label{tab1:HJ}
\end{table}

\begin{table}[!htbp]
\begin{center}
\begin{tabular}{|c||c|c|c|c|c|}
\hline
Type & $A$ & $B$ & $C$ & $D$ & $E$ \\
\hline
\hline
$AAA$ & $5$ & -- & -- & -- & -- \\
\hline
$ABB$ & $1$ & $1$ & -- & -- & -- \\
\hline
$ACC$ & $5$ & -- & $1$ & -- & -- \\
\hline
$BBB$ & -- & $5$ & -- & -- & -- \\
\hline
$BBC$ & -- & $10$ & $1$ & -- & -- \\
\hline
$CCC$ & -- & -- & $9$ & -- & -- \\
\hline
$CDD$ & -- & -- & $4$ & $25$ & -- \\
\hline
$DDD$ & -- & -- & -- & $6$ & -- \\
\hline
$DEE$ & -- & -- & -- & $1$ & $1$ \\
\hline
$EEE$ & -- & -- & -- & -- & $6$ \\
\hline 
\end{tabular}
\end{center}
\caption{The lines of the valuation geometry $\mathcal{V}$ of $\HJ$}
\label{tab2:HJ}
\end{table}

\begin{lem}
\label{lem:ValBC}
Let $f$ be a valuation of type $C$ and let $g \not= f$ and $h \not= f$ be valuations of type $B$ or $C$ lying on distinct $\mV$-lines through $f$. Then $g$ and $h$ are non-collinear.   
\end{lem}

\begin{lem}
\label{lem:ValuationGeometryHJ}
\label{lem:ValB}
Let $f$ be a valuation of type $B$ and let $x \in \mH$ be the unique point in $\mO_f$. Then the map $\{f, g, h\} \mapsto \mO_f \cup \mO_g \cup \mO_h$ is a bijection between the set of five $\mV$-lines of type $BBB$ through $f$ and the set of five lines of $\mH$ through $x$. 
\end{lem}

\begin{lem}
\label{lem:ValC}
Let $f$ be a valuation of type $C$. Then there is a unique $\mV$-line $\{f, g, h\}$ of type $CCC$ through $f$ such that $\mO_f \cup \mO_g \cup \mO_h$ is a line of $\HJ$. For every other $\mV$-line $\{ f,g',h' \}$ of type $CCC$ through $f$, the set $\mO_f \cup \mO_{g'} \cup \mO_{h'}$ is a set of three pairwise non-collinear points.
\end{lem}

A $\mV$-line $\{ f,g,h \}$ of type CCC will be called {\em special} if $\mO_f \cup \mO_g \cup \mO_h$ is a line of $\HJ$. If that is not the case, then $\{ f,g,h \}$ will be called an {\em ordinary} $\mV$-line. This concept of \textit{special} and \textit{ordinary} is then extended to the lines of $\mN$ that induce $\mV$-lines of type $CCC$. 

\begin{lem}
\label{lemG:connected}
The subgeometry of $\mV$ defined on the type $C$ valuations by the lines of type $ACC$ and the ordinary  lines of type $CCC$ is connected.
\end{lem}

\noindent \textit{\textbf{Remark}}: Lemmas \ref{lem:ValBC} and \ref{lem:ValC} could alternatively be verified by a geometric reasoning inside the $G_2(4)$ near octagon, keeping in mind its above mentioned construction using the valuations of type $A$, $B$ and $C$ of $\HJ$. 

\begin{lem} \label{lemG:basic}
Every point of $\mN$ is at distance at most $2$ from $\mH$. 
Points of $\mH$ are of type $A$, points at distance $1$ from $\mH$ are of type $B$ or $C$ and those at distance $2$ are of type $D$ or $E$. 
\end{lem}
\begin{proof}
Since $\mH$ is isometrically embedded in $\mN$, all points of $\mH$ induce type $A$ valuations. From Lemma \ref{new}(1) and the column $M_f$ of Table \ref{tab1:HJ} we see that the points in $\Gamma_1(\mH)$ cannot be of type $A$, but the points in $\Gamma_2(\mH)$ must be of type $D$ or $E$. If $x \in \Gamma_1(\mH)$, then there exists a line through $x$ that intersects $\mH$, which must necessarily be of type $ABB$ or $ACC$ by Table \ref{tab2:HJ}, implying that $x$ has type $B$ or $C$.
\end{proof}

\begin{cor} \label{corG:dist1}
Every point of $\mN$ at distance $1$ from $\mH$ is collinear with a unique point of $\mH$.
\end{cor}
\begin{proof}
Such points are of type $B$ or $C$ and valuations of type $B$ and $C$ have exactly one point of value $0$ (see column $|\mO_f|$ in Table \ref{tab1:HJ}). 
\end{proof}

\begin{lem} \label{lemG:noE}
There are no points of type $E$ in $\mN$.
\end{lem}
\begin{proof}
Let $x$ be a type $E$ point of $\mN$. By Lemma \ref{lemG:basic}, $x$ must be at distance $2$ from $\mH$. Let $y$ be a neighbor of $x$ which lies at distance $1$ from $\mH$. Then $x$ has type $B$ or $C$ by Lemma \ref{lemG:basic}. Since the valuations $f_x$ and $f_y$ are not equal, the line $xy$ gives rise to a $\mV$-line in the valuation geometry of $\HJ$ (see Lemma~\ref{lem:linetypes}). But, by Table \ref{tab2:HJ} there are no $\mV$-lines with both type $E$ and type $T$ points on it, for $T \in \{B, C\}$. 
\end{proof}

Let $x$ be a point of $\mN$ at distance $1$ from $\mH$ which by Lemma \ref{lemG:basic} is of type $B$ or $C$.
We will call the unique point of $\mH$ collinear with $x$ (see Corollary \ref{corG:dist1}) the \textit{projection} of $x$, and denote it by $\pi(x)$. 
From now onward we implicitly use the fact that points at distance $1$ from $\mH$ are of type $B$ or $C$.  
For a line $L = \{x, y, z\}$ contained in $\Gamma_1(\mH)$ we define the projection $\pi(L)$ of $L$ to be the set $\{\pi(x), \pi(y), \pi(x)\}$ of points of $\mH$. 
Since $\mN$ is a near polygon, $\pi(L)$ and $L$ have the same size for every line $L$ in $\Gamma_1(\mH)$. 
But, this projection may or may not be a line of $\mH$.

\begin{lem} \label{lemG:B}
Let $x$ be a type $B$ point of $\mN$ and let $y$ be a point on a line through $x$ which does not intersect $\mH$.
Then $y$ is at distance $1$ from $\mH$ and the projections $\pi(x)$ and $\pi(y)$ are collinear. 
\end{lem}
\begin{proof}
The point $y$ must be of type $B$ or $C$ since there are no $\mV$-lines containing both type $B$ and type $T$ points for $T \in \{D, E\}$ (see Table \ref{tab2:HJ}) and hence at distance $1$ from $\mH$ by Lemma \ref{lemG:basic}. 
The projections $\pi(x)$ and $\pi(y)$ have $f_x$-values $0$ and $1$, respectively. 
Since $f_x$ is of type $B$, there are exactly ten points of $\mH$ that have $f_x$-value $1$ (see Table \ref{tab1:HJ}).
Clearly, every point in $\mH$ at distance 1 from $\pi(x)$ has $f_x$-value $1$. 
Since $\mH$ has order $(2,4)$, there are precisely ten such points and hence $\pi(y)$ must be one of them. 
\end{proof}

\begin{cor}
\label{corG:B}
If $x$ is a point of type $B$ in $\mN$, then every line through $x$ that does not intersect $\mH$ is parallel to and at distance $1$ from a unique line of $\mH$. 
\end{cor}

\begin{lem}
\label{lemG:BBC}
Every type $B$ point $x$ of $\mN$ is incident with a line of type $BBC$. 
\end{lem}
\begin{proof}
Let $x$ be a point of type $B$. Then every point of $\mH$ at distance 1 from $\pi(x)$ has $f_x$-value 1. Since $\mO_{f_x} = \{ \pi(x) \}$ every point of $\mH$ at distance 2 from $\pi(x)$ should have $f_x$-value 2, and since $\mH$ is a regular near octagon with parameters $(2, 4; 0, 3)$, there are $80$ such points. By Table \ref{tab1:HJ} there are $112$ points of $\mH$ with $f_x$-value $2$. Let $y$ be one of the other $112-80=32$ points with $f_x$-value 2 at distance at least 3 from $\pi(x)$. Since $f_x(y) = 2$, we have $\dist(x, y) = 3$. Let $x, u, v, y$ be a path of length $3$ connecting $x$ and $y$. By Lemmas \ref{lemG:basic} and \ref{lemG:B}, the point $u$ has type $A$, $B$ or $C$. We will show that $u$ is of type $C$, hence proving that the line $xu$ is of type $BBC$. 

If $u$ is of type $A$, then $u = \pi(x)$, which would be in contradiction with $\dist(\pi(x), y) > 2$. Suppose $u$ is of type $B$, and hence at distance $1$ from $\mH$. From Corollary \ref{corG:B} we see that $\pi(u)$ and $\pi(x)$ are collinear (or equal). 
We cannot have $v = \pi(u)$ as that would imply that $\dist(\pi(x), y) \leq 2$. 
Therefore, $v$ lies outside $\mH$ and $y$ must be equal to $\pi(v)$. 
Again by Corollary \ref{corG:B} $y=\pi(v)$ and $\pi(u)$ must be collinear (or equal), which contradicts the fact that $\dist(\pi(x), y) > 2$. So, $u$ is of type $C$. 
\end{proof}

\begin{cor}
\label{corG:Cexists}
There exist type $C$ points in $\mN$.
\end{cor} 	
\begin{proof}
As there exist points at distance $1$ from $\mH$, there exist points of type $B$ or $C$. The existence of type $B$ points implies the existence of type $C$ points by Lemma \ref{lemG:BBC}. 
\end{proof}

\begin{lem}
\label{lemG:distinct_lines_CCC}
Let $x$ be a point of $\mN$ of type $C$ and let $L_1, L_2$ be two distinct lines of type $CCC$ through $x$. Then the $\mV$-lines corresponding to $L_1$ and $L_2$ must be distinct. 
\end{lem}
\begin{proof}
Let $L_1 = \{x, y, z\}$ and $L_2 = \{x, y', z'\}$. Assume that they correspond to the same $\mV$-line so that $f_y = f_{y'}$ and $f_z = f_{z'}$. 
Let $u := \pi(z) = \pi(z')$. Since $xzuz'$ is a quadrangle, the point $y$ must be collinear with the third point on the line $uz'$, call it $v$. 
Therefore, the valuations $f_v$ and $f_y = f_{y'}$ are collinear in $\mV$. 
The collinearity of $f_v$ and $f_{y'}$ in $\mV$ contradicts Lemma \ref{lem:ValBC} by taking $f = f_{z'}$, $g = f_{y'}$ and $h = f_v$. 
\end{proof}

On the valuations of type $C$ we can define a subgeometry of $\mV$ induced by the lines of type $ACC$ and the ordinary lines of type $CCC$.
Let this subgeometry be denoted by $\mV_C$ and its collinearity graph by $\Gamma_1$. 
Similarly, we can define a subgeometry $\mN_C$ of $\mN$ by taking the points of type $C$ and the lines that correspond to lines of $\mV_C$.
Let $\Gamma_2$ be the collinearity graph of  $\mN_C$.
Since type $C$ points exist in $\mN$, the graph $\Gamma_2$ is nonempty.

\begin{lem}
\label{lemG:cover}
The graph $\Gamma_2$ is a cover of the graph $\Gamma_1$ by the map $x \mapsto f_x$. 
\end{lem}
\begin{proof}
Let $x$ be a point of type $C$ in $\mN$. There are $16$ points of $f_x$-value $1$ that are not collinear with $\pi(x)$ (see Table \ref{tab1:HJ}). Denote this set of $16$ points by $U$. If $u \in U$, then $\dist(x,u) = 2$. Denote by $v$ a common neighbor of $x$ and $u$. Since $v \not= \pi(x)$, $v \not\in \mH$ and $u = \pi(v)$. Since $u \not= \pi(x)$, the point $v$ cannot be contained on the line $x\pi(x)$, and so $xv$ is a line disjoint from $\mH$. Since $\pi(x)$ and $\pi(v)=u$ are not collinear, Lemma \ref{lemG:B} implies that $v$ has type $C$ and hence that $xv$ is a (necessarily ordinary) line of type $CCC$.

By Lemmas \ref{lem:ValC}, \ref{lemG:distinct_lines_CCC} and Table \ref{tab2:HJ} there are at most $8$ ordinary lines of type $CCC$ through $x$, each of which determines two points of the set $U$.
Since $|U| = 16$, it follows that there are precisely $8$ ordinary lines of type $CCC$ through $x$ and they correspond bijectively to the $8$ ordinary $\mV$-lines of type $CCC$ through $f_x$.
This proves that the map $x \mapsto f_x$ is a local isomorphism between $\Gamma_2$ and $\Gamma_1$. 
The fact that this map is surjective now follows from the connectedness of $\Gamma_1$, see Lemma \ref{lemG:connected}. 
\end{proof}

\begin{cor}
\label{corG:cover}
If $\Gamma_2$ is an $i$-cover of $\Gamma_1$ for some $i \geq 1$, then each valuation of type $C$ is induced by precisely $i$ type $C$ points of $\mN$.
As a consequence, through each point of $\mH$, there are precisely $5i$ lines of type $ACC$. 
\end{cor}
\begin{proof}
By Table \ref{tab2:HJ} there are precisely $5$ $\mV$-lines of type $ACC$ through a given valuation of type $A$. 
Since each type $C$ valuation occurs exactly $i$ times, we have $5i$ type $ACC$ lines in $\mN$ through a given point of $\mH$. 
\end{proof}

\noindent
\textbf{\textit{Remark:}} All results in this section so far are valid for a general near octagon of order $(2,t)$ that contains an isometrically embedded sub near octagon isomorphic to $\HJ$. In the following lemma, we need the fact that $\mN$ has order $(2,10)$. 

\begin{lem}
\label{lemG:BC_exist_and_unique}
If $\mN$ is of order $(2,10)$, then each valuation of type $T \in \{B, C\}$ is induced exactly once by a point of $\mN$. 
\end{lem}
\begin{proof}
Let $\mN$ be of order $(2,10)$ and let $x$ be an arbitrary point of $\mH$.
Then there are exactly $11 - 5 = 6$ lines through $x$ that are not contained in $\mH$, each of which has type $ACC$ or $ABB$. 
Since type $C$ points exist by Corollary \ref{corG:Cexists}, it follows from Corollary \ref{corG:cover} that there are precisely $5$ lines of type $ACC$ through $x$ in $\mN$, and hence the graph $\Gamma_2$ is a $1$-cover of $\Gamma_1$. 
Now the $6$-th line through $x$ which is not contained in $\mH$ must be of type $ABB$. 
Therefore, through every point of $\mH$ there are $5$ lines of type $ACC$ and a unique line of type $ABB$. 
This shows that for every valuation $f$ of type $B$, we can find the unique point of $\mN$ that induces $f$ by first getting the point $y$ of $\mH$ that induces the type $A$ valuation on the unique $\mV$-line of type $ABB$ through $f$ (see Table \ref{tab2:HJ}), and then picking the point on the unique line of type $ABB$ through $x$ in $\mN$ that induces the valuation $f$.  
\end{proof}

For the rest of this section assume that $\mN$ has order $(2,10)$. 
From Lemma \ref{lemG:BC_exist_and_unique} we know that both type $B$ and type $C$ points exist in $\mN$ and each type $B$ or type $C$ valuation of $\mH$ is induced by a unique point of $\mN$.  
Let $x$ be a point of type $B$ in $\mN$ and let $L_x$ be the unique line joining $x$ and $\pi(x)$. 
From Corollary \ref{corG:B} it follows that every other line through $x$ gives rise to a quad in $\mN$ that intersects $\mH$ and contains $L_x$.

\begin{lem}
\label{lemG:QuadIntersection}
Let $Q$ be a quad of $\mN$ that intersects $\mH$ nontrivially.
Then $Q \cap \mH$ is either a singleton or a line.
\end{lem}
\begin{proof}
The proof is similar to that of Lemma \ref{lemHJ:QuadIntersection}.
\end{proof}

\begin{lem}
\label{lem:QuadBC}
Let $Q$ be a quad of $\mN$ that is not a grid and that intersects $\mH$ in a line $L$. Then there must exist points of type $B$ and points of type $C$ in $Q \setminus L$.
\end{lem}
\begin{proof}
For the sake of contradiction assume that all points of $Q \setminus L$ are of a fixed type $T \in \{B, C\}$. 
Let $x$ be a point of $L$.
Since $Q$ is not a grid, there exist two lines $L_1 = \{x, y, z\}$ and $L_2 = \{x, y', z'\}$ through $x$ with $y, y', z, z' \in Q \setminus L$. 
Let $w$ be a common neighbor of $z$ and $z'$ in $Q$ which is different from $x$. Then $w \in Q \setminus L$. 
From Lemma \ref{lemG:BC_exist_and_unique} it follows that the lines $wz$ and $wz'$ correspond to distinct $\mV$-lines, which are of type $TTT$ by our assumption. 
Also note that $\pi(wz) = \pi(wz') = L$. 
This contradicts Lemma \ref{lem:ValB} for $T = B$ and Lemma \ref{lem:ValC} for $T = C$.
\end{proof}

\begin{lem}
\label{lemG:noQ52}
\label{corG:noQ52}
There are no $Q(5,2)$-quads in $\mN$ that meet $\mH$ in a line.
\end{lem}
\begin{proof}
Let $Q$ be a $Q(5,2)$-quad that meets $\mH$ in a line $L$. By Lemma \ref{lem:QuadBC} there is a point $x$ of type $C$ in $Q$.
There is a unique line through $x$ that intersects $\mH$ in $L$, and hence lies in $Q$. 
Every other line through $x$ which is contained in $Q$ projects to $L$. 
By Lemmas \ref{lem:ValC} and \ref{lemG:distinct_lines_CCC} there is at most one line of type $CCC$ through $x$ in $Q$.
From Table \ref{tab2:HJ} and Lemma \ref{lemG:BC_exist_and_unique} it follows that there is at most one line of type $BBC$ through $x$ in $Q$.
Therefore, in total we have at most three lines through $x$ in $Q$ which contradicts the fact that the order of a $Q(5,2)$-quad is $(2,4)$. 
\end{proof}

\begin{lem}
\label{corG:BBB}
Let $x$ be a point of type $B$ in $\mN$. 
Then $x$ cannot be contained in two lines $L_1$, $L_2$ such that $L_1$ has type $BBB$, $L_2$ has type $BBC$ and $\pi(L_1) = \pi(L_2)$. 
\end{lem}
\begin{proof}
Let $L_1 = \{x, y, z\}$ and $L_2 = \{x, y', z'\}$ be two such lines, such that $\pi(y) = \pi(y')$ and $\pi(z) = \pi(z')$. 
Say $L_1$ is of type $BBC$ and $L_2$ of type $BBB$.
Without loss of generality assume that $z$ is of type $C$. 
Then the lines $y\pi(y)$ and $y'\pi(y')$ are of type $ABB$. 
By Table \ref{tab2:HJ} there is only one $\mV$-line of type $ABB$ through a valuation of type $A$, and hence $f_y$ is equal to $f_{y'}$ or $f_{y''}$ where $y''$ is the third point (of type $B$) on the line $y'\pi(y')$. 
This contradicts Lemma \ref{lemG:BC_exist_and_unique}.
\end{proof}

\begin{lem}
\label{lemG:typeB}
Let $x$ be a point of type $B$ in $\mN$.
Then
\begin{enumerate}[$(1)$]
\item  $x$ is incident with a unique line of type $ABB$ and ten lines of type $BBC$;
\item these ten type $BBC$ lines through $x$ correspond bijectively to the ten $BBC$ lines of the valuation geometry $\mV$ through $f_x$, and they are partitioned into pairs by five $W(2)$-quads passing through the line of type $ABB$ through $x$. 
\end{enumerate}
\end{lem}
\begin{proof}
There is a unique line through $x$ that intersects $\mH$, namely the line joining $x$ and $\pi(x)$.
Every other line through $x$ is of type $BBB$ or $BBC$ which is entirely contained in $\Gamma_1(\mH)$ and is parallel to a line through $\pi(x)$ in $\mH$ (see Corollary \ref{corG:B}). 
Let $S$ denote the set of these other lines through $x$.
By Lemma \ref{lemG:BC_exist_and_unique}, distinct lines in $S$ correspond to distinct $\mV$-lines. Let there be $i$ lines of type $BBB$ in $S$, with $i \leq 5$ by Table \ref{tab2:HJ}. 
Since we cannot have two lines of type $BBC$ and $BBB$ in $S$ projecting to the same line of $\mH$ by Lemma \ref{corG:BBB} and since there are no $Q(5,2)$-quads by Lemma \ref{lemG:noQ52}, there are at most $2(5-i)$ lines of type $BBC$ in $S$, and hence in total at most $2(5 - i) + i + 1 = 11 - i$ lines through $x$. 
Therefore, we have $i = 0$ and each of the $5$ lines of $\mH$ through $\pi(x)$  is parallel to exactly $2$ lines of $S$.
This gives rise to $5$ $W(2)$-quads through the line $x\pi(x)$, that partition $S$ into pairs. 
\end{proof}

We are now ready to prove Theorem \ref{thm:G24}. From Lemma \ref{lemG:typeB} it follows that there are no lines of type $BBB$ in $\mN$. 
Since each of the type $A$, $B$ and $C$ valuations is induced by a unique point of $\mN$ and each $\mV$-line of type $AAA$, $ABB$ and $ACC$ is induced by a unique line of $\mN$, it suffices to show that also every $\mV$-line of type $BBC$ and $CCC$ is induced by a unique line of $\mN$, and that type $D$ points do not exist in $\mN$ (we have already proved in Lemma \ref{lemG:noE} that type $E$ points do not exist). 

Let $\{ f, g, h \}$ be a $\mV$-line of type $BBC$ where $f$ is of type $B$. Let $x$ be the unique point in $\mN$ with $f_x = f$. By Lemma \ref{lemG:typeB}, there exists a line $L = \{x, y, z\}$ such that $f_y = g$ and $f_z = h$. This shows that each $\mV$-line of type $BBC$ is induced by a necessarily unique line of $\mN$. 

Now, let $x$ be a point of type $C$. Since $\Gamma_2$ is a $1$-cover of $\Gamma_1$, there exist eight ordinary lines of type $CCC$ through $x$ that bijectively correspond to the eight ordinary $\mV$-lines of type $CCC$ through $f_x$. By Table \ref{tab2:HJ}, there exists a unique $\mV$-line of type $BBC$ through $f_x$, implying that in $\mN$ there is a unique line $L$ of type $BBC$ through $x$. By Lemma \ref{lemG:typeB} $L$ lies in a $W(2)$-quad $Q$, which must also contain the unique line of type $ACC$ through $x$. The third line in $Q$ through $x$ must be a special line of type $CCC$ as there is a unique type $BBC$ line through $x$ and none of the ordinary type $CCC$ lines through $x$ projects to a line of $\mH$. Therefore, the unique special $\mV$-line of type CCC through $f_x$ is induced by a line of $\mN$. Since we have accounted for all $11$ lines through a point of type $C$, there cannot be any lines of type $CDD$, and hence there cannot be any points of type $D$ in $\mN$. This completes the proof as we have shown that $\mN$ is isomorphic to the $G_2(4)$ near octagon.

\end{document}